\documentclass[11pt]{amsart}

\usepackage{amsmath}
\usepackage{amssymb}
\usepackage{graphicx}

\newtheorem{theorem}{Theorem}[section]
\newtheorem{corollary}[theorem]{Corollary}
\newtheorem{lemma}[theorem]{Lemma}
\newtheorem{proposition}[theorem]{Proposition}
\theoremstyle{definition}
\newtheorem{definition}[theorem]{Definition}
\newtheorem{example}[theorem]{Example}
\newtheorem{remark}[theorem]{Remark}

\numberwithin{equation}{section}

\title[Set-inclusive generalized equations via convex analysis]
{Solution analysis for a class of set-inclusive
generalized equations: a convex analysis approach}

\author[A. Uderzo]{Amos Uderzo}

\address[A. Uderzo]{Dept. of Mathematics and Applications, University
of Milano - Bicocca, Milano, Italy}
\email{{\tt amos.uderzo@unimib.it}}



\date{\today}

\newcommand{\R}{\mathbb R}
\newcommand{\N}{\mathbb N}

\newcommand{\Z}{\mathbb Z}

\newcommand{\X}{\mathbb X}
\newcommand{\Y}{\mathbb Y}
\newcommand{\Uball}{{\mathbb B}}
\newcommand{\Usfer}{{\mathbb S}}

\newcommand{\dom}{{\rm dom}\, }

\newcommand{\nullv}{\mathbf{0}}

\newcommand{\wclco}{\overline{\rm conv}{\,}^*\, }

\newcommand{\inte}{{\rm int}\, }
\newcommand{\weakstar}{weak${}^*\, $ }

\newcommand{\Lin}{\mathcal{L}}
\newcommand{\Conv}{{\mathcal C}}
\newcommand{\BConv}{\mathcal{BC}}

\newcommand{\sur}{{\rm sur}\, }

\newcommand{\GE}{{\rm GE}\,}
\newcommand{\IGE}{{\rm IGE}\,}
\newcommand{\VOP}{{\rm VOP}\,}

\newcommand{\parord}{\le_{{}_C}}

\newcommand{\stsl}[1]{|\nabla #1|}
\newcommand{\Solv}[1]{{\mathcal Sol}(#1)}
\newcommand{\dcone}[1]{{#1}^{{}^\ominus}}
\newcommand{\erbom}[1]{\tau_{#1}}
\newcommand{\derbom}[1]{\dcone{\tau_{#1}}}
\newcommand{\charfan}[1]{\varphi_{#1}}
\newcommand{\Bconst}[1]{{\flat}(#1)}
\newcommand{\dBconst}[1]{{\flat}^*(#1)}

\newcommand{\ball}[2]{{\rm B}(#1, #2)}

\newcommand{\dist}[2]{{\rm dist}\left(#1,#2\right)}

\newcommand{\Tang}[2]{{\rm T}(#1;#2)}

\newcommand{\supf}[2]{\varsigma({#1,#2})}
\newcommand{\charfun}[2]{\varphi_{#1,#2}}
\newcommand{\dcharfun}[2]{\varphi_{#1,#2}^{{}^\ominus}}
\newcommand{\erboc}[2]{|\partial{#1}|{(#2)}}
\newcommand{\derboc}[2]{|\dcone{\partial}{#1}|{(#2)}}

\begin{document}

\begin{abstract}
In the present paper, classical tools of convex analysis are used to study
the solution set to a certain class of set-inclusive generalized
equations. A condition for the solution existence and global error bounds is
established, in the case the set-valued term appearing in the generalized equation is concave.
A functional characterization of the contingent cone to the solution set is provided
via directional derivatives. Specializations of these results are also considered
when outer prederivatives can be employed.
\end{abstract}

\maketitle


\begin{flushright}
{\small
 ``The value of convex analysis is still not \\ fully appreciated (by the non
 specialists)", \\
J. M. Borwein, \cite{Borw17}}
\end{flushright}

\vskip1cm


\section{Introduction}

The term ``generalized equation" denotes a widely recognized
format for modeling a broad variety of problems arising in
optimization and variational analysis. the successful employment
of such a format rests upon its main distinguishing feature, namely
the capability of involving inclusions, multi-valued mappings and
sets. Indeed, inclusions (or, more generally, one-side relations),
multi-valued mappings and sets (the latter ones, handled as a whole)
are the basic elements on which the modern theory of optimization
and variational analysis is built.

The type of generalized equation mainly studied in the last decades
is of the form
$$
   \hbox{find $x\in S$ such that $\nullv\in f(x)+F(x)$},
   \leqno (\GE)
$$
where $f:\X\longrightarrow\Y$ and $F:\X\rightrightarrows\Y$ are
given single-valued and set-valued mappings, respectively, and
$S\subseteq\X$ and $C\subseteq\Y$ are given subsets of vector spaces,
with null element $\nullv$. Such a format was distilled as a unifying device
to cover traditional equality/inequality systems, occurring as
constraints in mathematical programming problems, as well as
variational inequalities (and hence, complementarity problems)
differential inclusions, coincidence (and hence, fixed point)
problems, optimality (included Lagrangian) conditions for
variously constrained optimization problems.

The present paper deals instead with generalized equations of a different
form, namely
$$
   \hbox{find $x\in S$ such that $F(x)\subseteq C$},
   \leqno (\IGE)
$$
where $F:\X\rightrightarrows\Y$ is a set-valued mapping
between Banach spaces, $C\subseteq\Y$ a (nonempty) closed, convex
set and $S\subseteq\X$. Generalized equations of this type will
be called ``set-inclusive".

Generalized equations like $(\IGE)$ have been so far less investigated than
$(\GE)$, for which a well-developed theory is now at disposal
(see, among others, \cite{DonRoc14,Izma14,Mord94,Mord06,Robi79,Robi80,Robi83}).
Nevertheless, there are several contexts in which the format of set-inclusive
generalized equations does emerge. Some of these contexts are illustrated
below.

\vskip.5cm

\noindent{\it 1. Robust approach to uncertain constraint systems}:
Let us consider a cone constraint system formalized by the parametric
inclusion
\begin{equation}     \label{in:parcontsys}
  f(x,\omega)\in C,
\end{equation}
where $f:\R^n\times\Omega\longrightarrow\R^m$ is a given mapping
and $C$ is a closed, convex cone in $\R^m$. For instance,
if $\R^m=\{\nullv\}\times\R^q_-$, with $\nullv\in\R^p$ and
$p+q=m$, then $(\ref{in:parcontsys})$ turns out to represent
a system of finitely many equalities and inequalities, which
is a typical constraint system in mathematical programming.
The parameter $\omega\in\Omega$ entering the argument of $f$
describes uncertainties often occurring in real-world optimization
problems. In fact, the feasible region of such problems,
as well as their objective function, may happen to be affected by
computational and estimation errors, and conditioned by unforeseeable
future events. Whereas a stochastic optimization approach
requires the probability distribution of the uncertain parameter
to appear among the problem data, robust optimization assumes that
no stochastic information on the uncertain parameter is at disposal.
This opens the question on what can be admitted as a solution to
the system $(\ref{in:parcontsys})$, in consideration of possible
outcomes depending on the parameter $\omega$.
According to the robust approach, an element $x\in\R^n$ is considered
to be a feasible solution if it remains feasible in every possibly occurring
scenario, i.e. if it is such that
$$
   f(x,\omega)\in C,\quad\forall \omega\in\Omega.
$$
Such an approach naturally leads to introduce the robust constraining
mapping $F:\R^n\rightrightarrows\R^m$, defined as
\begin{equation}     \label{eq:robconstmap}
  F(x)=f(x,\Omega)=\{f(x,\omega):\ \omega\in\Omega\},
\end{equation}
and to consider set-inclusive generalized equations like $(\IGE)$.

\vskip.5cm

\noindent{\it 2. Ideal solutions in vector optimization}:
Let $f:\X\longrightarrow\Y$ be a function taking values in a
vector space $\Y$ partially ordered by its (positive) cone
$\Y_+$ and let $R\subseteq\X$ be a nonempty set.
Recall that $\bar x\in R$ is said to be an ideally $\Y_+$-efficient
solution for the related vector optimization problem
$$
  \Y_+\hbox{-}\min f(x) \quad\hbox{ subject to }\qquad x\in R,
  \leqno (\VOP)
$$
provided that
$$
   f(R)\subseteq f(\bar x)+\Y_+.
$$
Thus, by introducing the set-valued mapping $F:\X\rightrightarrows\Y$
defined by $F(x)=f(R)-f(x)$, one gets that the set of all ideally
$\Y_+$-efficient solutions coincides with the solution set of a
set-inclusive generalized equation as $(\IGE)$, with $C=\Y_+$. It is worth
recalling that any ideal $\Y_+$-efficient solution is, in particular,
also $\Y_+$-efficient (for more details on optimality notions in vector
optimization and their relationships, see \cite{Jahn04}).

\vskip.5cm

\noindent{\it 3. Constraints on production in mathematical economics}: In
mathematical economics, a production technology, i.e. the description of
quantitative relationships between inputs and outputs, can be conveniently
formalized by a set-valued mappings $F:\R^n\rightrightarrows\R^m$,
associating with an output $x\in\R^n$ the set of all inputs $y\in\R^m$
which are needed to produce $x$, according to the technology at the issue
(meaning that the same output can be obtained by combining inputs in
different ways). In other terms, if $x$ is seen as a (vector) production
level, $F(x)$ represents the corresponding isoquant (see \cite{Fare88}).
In this setting, given a closed subset $C\subseteq\R^m$, a set-inclusive
generalized equation $(\IGE)$ describes the presence of constraints,
due to specific requirements on the input employment, which are
not intrinsic to the production technology itself.

\vskip.5cm

In the absence of an ad hoc theory,
the investigations exposed in the present paper aim at providing
elements for a solution analysis of $(\IGE)$. More precisely,
they focus on solvability and global error bound conditions
for a $(\IGE)$ and, by means of them, they leads to obtain
first-order approximations of its solution set.
Apart from the very recent paper \cite{Uder19},
to the best of the author's knowledge, up to now generalized equations
in the form $(\IGE)$ have been considered only in \cite{Cast99},
where, nonetheless, the solution existence is taken as an assumption
in order to establish an error bound result.
In the same vein as in \cite{Cast99}, in the current study the
task is undertaken by using tools and techniques of convex analysis.
In doing so, the author, who ascribes himself to the class of non
specialists of convex analysis, would like to make an attempt
to contrast the phenomenon signaled by J.M. Borwein (see the quotation
put as an incipit for the present paper).

Whereas in \cite{Uder19} the problem is addressed by introducing
the metric $C$-increase property, 
the main idea behind the analysis here proposed is borrowed, with
some modifications, from \cite{Cast99}. It relies on the use of the
Minkowski-H\"ormander duality for passing from relations
between closed convex sets to corresponding relations between
convex functions. This passage is actually the key step, paving the
way to a functional characterization of solutions to $(\IGE)$.
This, in turn, triggers well-known techniques now at disposal
in variational analysis for treating such issues as solvability and
error bounds of convex inequalities. Such an approach can be said to
act in accordance with the celebrated Euler's spirit: indeed, solutions
to $(\IGE)$ are regarded as minimizers of certain functionals.
The fundamental assumptions allowing one to conduct the aforementioned
analysis, while remaining within the realm of convex analysis,
is the concavity of the set-valued mapping $F$ and the convexity of
the subset $C$. It seems that the former one has not yet found great
application in variational analysis, even if it must be said that, in a
special case, it already appeared, at the very initial stage
of nonsmooth analysis, within the theory of fans (see Example
\ref{ex:fan}). In fact, the concavity of fans will be exploited here
to specialize the main results, when outer prederivatives are
at disposal.

The contents of the paper are arranged in the subsequent sections
as follows. Section \ref{Sect:2} collects the essential technical preliminaries:
basic elements of convex and variational analysis are recalled, the
crucial notion of concavity for set-valued mappings is discussed
through several examples, some ancillary results are derived.
In Section \ref{Sect:3} the main results of the paper are exposed:
the first one is a sufficient condition for the solvability of
a $(\IGE)$ with a related error bound, while the second is a functional
characterization of the contingent cone to the solution set.
Section \ref{Sect:4} complements the previous section by providing
an estimate of the constant, appearing in the aforementioned findings,
with tools of set-valued analysis.

\vskip1cm


\section{Tools of analysis}   \label{Sect:2}

The notations in use throughout the paper are mainly standard.
Quite often, capital letter in bold will denote real Banach spaces.
$\Conv(\Y)$ denotes the class of all closed and convex subsets of
a Banach space $\Y$, while $\BConv(\Y)$ its subclass consisting of
all bounded, closed and convex sets.
The null vector in a Banach space is denoted by $\nullv$.
In a metric space setting, the closed  ball centered at an element $x$,
with radius $r\ge 0$, is indicated with $\ball{x}{r}$. In particular,
in a Banach space, $\Uball=\ball{\nullv}{1}$, whereas $\Usfer$ stands
for the unit sphere.
Given a subset $S$ of a Banach space, $\inte S$ denotes its interior.
The distance of a point $x$ from $S$ is denoted by $\dist{x}{S}$.
By $\Lin(\X,\Y)$ the Banach space of all bounded linear operators
acting between $\X$ and $\Y$ is denoted, equipped with the operator norm
$\|\cdot\|_\Lin$. In particular, $\X^*=\Lin(\X,\R)$ stands for the dual
space of $\X^*$, in which case $\|\cdot\|_\Lin$ is simply marked by $\|\cdot\|$.
The null vector, the unit ball and the unit sphere in a dual space
will be marked by $\nullv^*$, $\Uball^*$, and $\Usfer^*$, respectively.
The duality pairing a Banach space with its dual will be denoted
by $\langle\cdot,\cdot\rangle$. If $S$ is a subset of a dual space,
$\wclco S$ stands for its convex closure with respect to the
\weakstar topology.
Given a function $\varphi:\X\longrightarrow\R\cup\{\mp\infty\}$,
by $[\varphi\le 0]=\varphi^{-1}((-\infty,0])$ its sublevel set is
denoted, whereas $[\varphi>0]=\varphi^{-1}((0,+\infty))$ denotes
the strict superlevel set of $\varphi$.
The acronyms l.s.c., u.s.c. and p.h. stand for lower semicontinuous, upper
semicontinuous and positively homogeneous, respectively.
The symbol $\dom\varphi$ indicates the domain of the function $\varphi$.
The solution set to $(\IGE)$ is denoted by $\Solv{\IGE}$.

\subsection{Convex analysis tools}

The approach of analysis here proposed is strongly based on
the employment of the support function associated with an
element of $\Conv(\Y)$, henceforth denoted by $\supf{\cdot}{C}
:\Y^*\longrightarrow\R\cup\{\pm\infty\}$, namely the function
defined by
$$
   \supf{y^*}{C}=\sup_{y\in C}\langle y^*,y\rangle,
$$
where, consistently with the convention $\sup\varnothing=-\infty$, it is
$\supf{\cdot}{\varnothing}=-\infty$.
In this perspective, the following remark gathers some basic
well-known properties of support functions that will be exploited
in the sequel (see, for instance, \cite{Zali02}).

\begin{remark}    \label{rem:supfprops}
(i) For any $C\in\Conv(\Y)$, $\supf{\cdot}{C}$ is a (norm) l.s.c., p.h.
convex (sublinear) function on $\Y^*$. Furthermore, $\supf{\cdot}{C}$
is also l.s.c. with respect to the \weakstar topology on $\Y^*$.

(ii) Let $C,\, D\in\Conv(\Y)$ and let $\lambda,\, \mu$ be nonnegative
reals. Then, it holds
$$
  \supf{\cdot}{\lambda C+\mu D}=\lambda\supf{\cdot}{C}+
  \mu\supf{\cdot}{D}.
$$

(iii) Let $C,\, D\in\Conv(\Y)$. Then, it holds
$$
  C\subseteq D \qquad \hbox{ iff }\qquad
  \supf{y^*}{C}\le\supf{y^*}{D},\quad\forall y^*\in\Y^*.
$$
It is relevant to add that such a characterization of the inclusion
$ C\subseteq D$ holds true even if $\Y^*$ is replaced with $\Uball^*$,
as the support function is p.h. (remember point (i) in the current
remark).

(iv) If, in particular, it is $C\in\BConv(\Y)\backslash\{\varnothing\}$,
then $\supf{\cdot}{C}:\Y^*\longrightarrow\R$ is (Lipschitz)
continuous on $\Y^*$.

(v) Let $C\in\Conv(\X^*)$ and consider its support $\supf{\cdot}{C}:
\X\longrightarrow\R\cup\{\pm\infty\}$, i.e. $\supf{x}{C}=\sup_{x^*\in C}
\langle x^*,x\rangle$. Then $\nullv^*\in C$ iff $[\supf{\cdot}{C}\ge
0]=\X$. More precisely, the following estimate is valid
$$
   -\inf_{v\in\Uball}\supf{v}{C}\le\dist{\nullv^*}{C}.
$$
Indeed, one has
\begin{eqnarray*}
   -\inf_{v\in\Uball}\supf{v}{C} &=&\sup_{v\in\Uball}\inf_{x^*\in C}
   \langle x^*,-v\rangle=\sup_{v\in\Uball}\inf_{x^*\in C}
   \langle x^*,v\rangle\le\inf_{x^*\in C}\sup_{v\in\Uball}
   \langle x^*,v\rangle  \\
   &=& \inf_{x^*\in C} \|x^*\|=\dist{\nullv^*}{C}.
\end{eqnarray*}
\end{remark}

Following a line of though well recognized in the literature on
the subject (see \cite{FaHeKrOu10} and references therein),
the main condition for achieving solvability and error bounds
for $(\IGE)$ will be expressed in dual terms, namely by means of
constructions in the space $\X^*$, involving the subdifferential
in the sense of convex analysis. Recall
that, given a convex function $\varphi:\X\longrightarrow\R
\cup\{+\infty\}$ and $x_0\in\dom\varphi$, its subdifferential
$\partial\varphi(x_0)$ at $x_0$ is defined as
$$
  \partial\varphi(x_0)=\{x^*\in\X^*:\ \langle x^*,x-x_0\rangle
  \le\varphi(x)-\varphi(x_0),\quad\forall x\in\X\}.
$$
As convex functions may happen to be nonsmooth, such a
tool of analysis can be regarded as a surrogate of a
derivative, whenever the latter fails to exist. Therefore,
the mentioned condition for solvability and error bound
can be said also to be of infinitesimal type.

\begin{remark}   \label{rem:DubMil}
The following subdifferential calculus rule, which is a generalization
to compact index sets of the well-known Duboviskii-Milyutin rule,  will
be applied in subsequent arguments: let $\Xi$ be a
separated compact topological space and let $\varphi:\Xi\times\X
\longrightarrow\R$ be a given function. Suppose that:

(i) the function $\xi\mapsto\varphi(\xi,x)$ is u.s.c. on $\Xi$,
for every $x\in\X$;

(ii) the function $x\mapsto\varphi(\xi,x)$ is convex and
continuous at $x_0\in\X$, for every $\xi\in\Xi$.

\noindent Under the above assumptions, by introducing the
(clean-up) subset
$\Xi_{x_0}=\{\xi\in\Xi:\ \varphi(\xi,x_0)=\sup_{\xi\in\Xi}\varphi
(\xi,x_0)\}$, it results in
$$
  \partial\left(\sup_{\xi\in\Xi}\varphi(\xi,\cdot)\right)(x_0)
  =\wclco\left(\bigcup_{\xi\in\Xi_{x_0}}\partial\varphi(\xi,\cdot)
  (x_0)\right)
$$
(see \cite[Theorem 2.4.18]{Zali02}).
\end{remark}

The directional derivative of a function $\varphi:\X\longrightarrow
\R\cup\{+\infty\}$ at $x_0\in\dom\varphi$ in the direction $v\in\X$
is denoted by $\varphi'(x_0;v)$. Recall that whenever $\varphi$ is
a convex function continuous at $x_0$, then $\partial\varphi(x_0)$
is a nonempty, \weakstar compact convex subset of $\X^*$ and the following
Moreau-Rockafellar representation formula holds
\begin{equation}     \label{eq:MorRoc}
  \varphi'(x_0;v)=\supf{v}{\partial\varphi(x_0)},\quad\forall
  v\in\X
\end{equation}
(see \cite[Theorem 2.4.9]{Zali02}).

The special class of $(\IGE)$, for which the solution analysis
will be carried out, is singled out by a geometric property of
the set-valued mapping $F$ appearing in $(\IGE)$. Such a property,
which is introduced next under the term concavity, has merely to do
with the vector structure of the spaces $\X$ and $\Y$.

\begin{definition}     \label{def:concavesvmap}
A set-valued mapping $F:\X\rightrightarrows\Y$ between Banach spaces
is said to be {\it concave} on $\X$ if it holds
\begin{equation}     \label{def:concavpro}
  F(tx_1+(1-t)x_2))\subseteq tF(x_1)+(1-t)F(x_2),\quad
  \forall x_1,\, x_2\in\X,\ \forall t\in [0,1].
\end{equation}
\end{definition}

\begin{remark}
Whereas the notion of convexity for set-valued mappings is
equivalent to the convexity of their graph, thereby entailing
remarkable properties on their behaviour (e.g. a convex multivalued
mapping takes always convex values and carries convex sets into
convex sets, their inverse is still convex, and so on
\footnote{For a view on properties of
convex set-valued mappings of interest in optimization, the reader
is referred to \cite{Borw81}.}),
this fails generally to be true for the notion of concavity,
as proposed in Definition \ref{def:concavesvmap}. For instance,
the mapping $F:\R\rightrightarrows\R$ defined by $F(x)=\{-1,\, 1\}$
for every $x\in\R$, fulfils Definition \ref{def:concavesvmap},
but its values are not convex, for every $x\in\R$.
\end{remark}

Below, some circumstances in which the property of concavity
for set-valued mappings emerges are presented.

\begin{example}   \label{ex:concsvmap}
(i) Let $\varphi:\X\longrightarrow\R$ be a convex function. Then,
it is possible to show that the (hypographical) set-valued mapping
${\rm Hyp}_\varphi:\X\rightrightarrows\R$, defined by
$$
  {\rm Hyp}_\varphi(x)=\{r\in\R:\ r\le\varphi(x)\},
$$
is concave.

(ii) In a similar manner, it is possible to show that is concave
the (epigraphical) set-valued mapping ${\rm Epi}_\psi:\X\rightrightarrows\R$,
defined by
$$
  {\rm Epi}_\psi(x)=\{r\in\R:\ r\ge\psi(x)\},
$$
provided that $\psi:\X\longrightarrow\R$ is a concave function.

(iii) By combining what observed in (i) and (ii) one gets that
the set-valued mapping $F:\X\rightrightarrows\R$ defined by
$$
  F(x)=\{r\in\R:\ \psi(x)\le r\le\varphi(x)\},
$$
with $\psi(x)\le\varphi(x)$ for every $x\in\X$, is concave on $\X$.

(iv) Let $\Y$ be a Banach space endowed with a partial ordering
$\parord$, defined by a closed, convex cone $C\subseteq\Y$, and let
$f:\X\longrightarrow\Y$ be a $C$-convex mapping, i.e. any mapping
satisfying the condition
$$
  f(tx_1+(1-t)x_2))\parord tf(x_1)+(1-t)f(x_2),\quad
  \forall t\in [0,1],\ \forall x_1,\, x_2\in\X
$$
(for more on this class of mappings, see \cite{Borw82}).
Then, the set-valued mapping ${\rm Hyp}_f:\X\rightrightarrows\Y$,
defined by
$$
  {\rm Hyp}_f(x)=\{y\in\Y:\ y\parord f(x)\}
$$
is concave on $\X$. To see this fact, take arbitrary $x_1,\, x_2\in\X$
and $t\in [0,1]$, and let $y$ be an arbitrary element in the set
${\rm Hyp}_f(tx_1+(1-t)x_2)$. Since it holds
$$
  y\parord f(tx_1+(1-t)x_2)\parord tf(x_1)+(1-t)f(x_2),
$$
then, by setting $c=tf(x_1)+(1-t)f(x_2)-y\in C$, one can write
\begin{equation}    \label{eq:CconfconcaveF}
    y=t(f(x_1)-c)+(1-t)(f(x_2)-c).
\end{equation}
By observing that
$$
  f(x_1)-c\in {\rm Hyp}_f(x_1)\qquad\hbox{ and }\qquad
  f(x_2)-c\in {\rm Hyp}_f(x_2),
$$
equality $(\ref{eq:CconfconcaveF})$ says that $y\in t{\rm Hyp}_f
(x_1)+(1-t){\rm Hyp}_f(x_2)$, thereby showing that inclusion
$(\ref{def:concavpro})$ happens to be satisfied. Notice that, taking
$\Y=\R$ with $C=[0,+\infty)$, example (iv) subsumes example (i).
\end{example}

\begin{example}[Radial mapping]     \label{ex:radconcmap}
Given a convex function $\rho:\X\longrightarrow [0,+\infty)$, let
$F:\X\rightrightarrows\Y$ be defined by
$$
  F(x)=\rho(x)\Uball=\ball{\rho(x)}{\nullv},
$$
where $\Uball$ stands here for the unit ball of the space $\Y$.
It is readily seen that $F$ is a concave set-valued mapping.
\end{example}

\begin{example}[Fan]     \label{ex:fan}
After \cite{Ioff81}, a set-valued mapping $A:\X\rightrightarrows\Y$
between Banach spaces is said to be a fan if all the following
conditions are fulfilled:

(i) $\nullv\in\ A(\nullv)$;

(ii) $A(\lambda x)=\lambda A(x)$, $\forall x\in\X$ and $\forall\lambda>0$;

(iii) $A(x)\in\Conv(\Y)$, $\forall x\in\X$;

(iv) $A(x_1+x_2)\subseteq A(x_1)+A(x_2)$, $\forall x_1,\, x_2\in\X$.

\noindent Owing to conditions (ii) and (iv), it is clear that any fan is
a (p.h.) concave set-valued mapping. As a particular example of fan,
one can consider set-valued mappings which are generated by families
of linear bounded operators. More precisely, let $\mathcal{G}\subseteq\Lin
(\X,\Y)$ be a convex set weakly closed with respect to the weak topology
on $\Lin(\X,\Y)$ and let
$$
  A_\mathcal{G}(x)=\{y\in\Y:\ y=\Lambda x,\, \Lambda\in\mathcal{G}\}.
$$
The set-valued mapping $A_\mathcal{G}:\X\rightrightarrows\Y$ is known to
be a particular example of fan (note however that there are fans which can not
be generated by families of linear bounded operators).

Notice that, if in Example \ref{ex:concsvmap}(iv) the mapping $f$ is
assumed to be also p.h., the resulting hypographical set-valued mapping
${\rm Hyp}_f$ turns out to be a fan. The same if in Example \ref{ex:radconcmap}
function $\rho$ is assumed to be sublinear on $\X$.

Fans may be employed in the robust approach to the uncertain constraint
system analysis. Let $\Omega$ be an arbitrary set of parameters
and let $p:\Omega\longrightarrow\Lin(\X,\Y)$ be a given mapping,
such that $p(\Omega)$ is a weakly closed and convex subset of
$\Lin(\X,\Y)$. Consider the mapping $f:\X\times\Omega\longrightarrow
\Y$ defined as
$$
  f(x,\omega)=p(\omega)x,
$$
which formalizes a uncertain constraint system of the type $(\ref{in:parcontsys})$.
Following the robust approach, one has to handle the set-valued mapping
$F:\X\rightrightarrows\Y$ given by
$$
  F(x)=f(x,\Omega)=\{y\in\Y:\ y=\Lambda x,\, \Lambda\in p(\Omega)\}.
$$
As a fan, $F$ turns out to be a concave mapping on $\X$. It is worth
noting that, whenever the set $p(\Omega)$ is $\|\cdot_\|\Lin$-bounded,
$F$ takes nonempty closed, convex and bounded values.
\end{example}

It is plain to see that if $F:\X\rightrightarrows\Y$ and $G:\X
\rightrightarrows\Y$ are concave on $\X$, so are $F+G$ and $\lambda
F$, for every $\lambda\in\R$. If $H:\X\rightrightarrows\Z$ is a
concave set-valued mapping between Banach spaces, so is the
Cartesian product mapping $F\times G:\X\rightrightarrows\Y\times\Z$,
defined by $(F\times G)(x)=F(x)\times G(x)$. Furthermore, if
$\Lambda\in\Lin(\Z,\X)$, then the set-valued mapping $F\circ\Lambda
:\Z\rightrightarrows\Y$ is still concave. Instead, if $F:\X
\rightrightarrows\Y$ is concave, its inverse set-valued mapping
$F^{-1}:\Y\rightrightarrows\X$ generally fails to be so.

Given a generalized equation in the form $(\IGE)$, according to
the approach here proposed, the functions $\charfun{F}{C}:\X
\longrightarrow\R\cup\{+\infty\}$ and $\dcharfun{F}{C}:\X
\longrightarrow\R\cup\{+\infty\}$ defined as follows will
play a crucial role as a basic tool of analysis:
\begin{equation}    \label{eq:charfundef}
  \charfun{F}{C}(x)=\sup_{b^*\in\Uball^*}[\supf{b^*}{F(x)}-\supf{b^*}{C}]
\end{equation}
and
\begin{equation}    \label{eq:dcharfundef}
  \dcharfun{F}{C}(x)=\sup_{b^*\in\Uball^*\cap\dcone{C}}
  [\supf{b^*}{F(x)}-\supf{b^*}{C}].
\end{equation}
In the lemma below some useful properties of $\charfun{F}{C}$ and
$\dcharfun{F}{C}(x)$ are deduced from assumptions on $F$ and $C$.

\begin{lemma}    \label{lem:charfunprops}
Let $F:\X\rightrightarrows\Y$ be a set-valued mapping
between Banach spaces.

(i) If $F(x)\in\BConv(\Y)\backslash\{\varnothing\}$ for every
$x\in\X$, then $\charfun{F}{C}$ is a nonnegative and real-valued
function, i.e. $\dom\charfun{F}{C}=\X$;

(ii) If $F$ is concave on $\X$, then $\charfun{F}{C}$ is
convex on $\X$;

(iii) If $F$ is p.h. and $C$ is a cone, then $\charfun{F}{C}$ is p.h..
\end{lemma}

\begin{proof}
(i) First of all observe that, independently of the boundedness
assumption, one has by definition
$$
  \charfun{F}{C}(x)\ge\supf{\nullv^*}{F(x)}-\supf{\nullv^*}{C}=0,
  \quad\forall x\in\X,
$$
so $\charfun{F}{C}$ takes nonnegative values only (and hence,
is bounded from below). Now, fix an arbitrary
$x\in\X$ and, according to the assumption, suppose that there exists $\kappa>0$
such that $F(x)\subseteq\kappa\Uball$. By recalling Remark \ref{rem:supfprops}
(iii) and (ii), one finds for every $b^*\in\Uball^*$
$$
  \supf{b^*}{F(x)}\le\supf{b^*}{\kappa\Uball}=\kappa\supf{b^*}{\Uball}=
  \kappa.
$$
If $c_0\in C$, one has
$$
  \supf{b^*}{C}\ge\langle b^*,c_0\rangle\ge -\|c_0\|,
  \quad\forall b^*\in\Uball^*,
$$
wherefrom it follows
$$
  \inf_{b^*\in\Uball^*}\supf{b^*}{C}\ge -\|c_0\|.
$$
Consequently, one obtains
$$
  \charfun{F}{C}(x)\le \sup_{b^*\in\Uball^*}\supf{b^*}{F(x)}-
  \inf_{b^*\in\Uball^*}\supf{b^*}{C}\le\kappa+\|c_0\|<+\infty.
$$

(ii) Let $x_1,\, x_2\in\X$ and $t\in [0,1]$. According to the
assumption of the concavity on $F$, inclusion $(\ref{def:concavpro})$
holds true. By recalling Remark \ref{rem:supfprops}(iii), that
inclusion implies
\begin{equation*}
    \supf{b^*}{F(tx_1+(1-t)x_2)}\le  \supf{b^*}{tF(x_1)+
    (1-t)F(x_2)},\quad\forall b^*\in\Uball^*.
\end{equation*}
From this inequality, by using the equalities in Remark
\ref{rem:supfprops}(ii), one readily sees
\begin{equation*}
    \supf{b^*}{F(tx_1+(1-t)x_2)}\le  t\supf{b^*}{F(x_1)}+
    (1-t)\supf{b^*}{F(x_2)},\quad\forall b^*\in\Uball^*,
\end{equation*}
which shows the convexity of the function $x\mapsto \supf{b^*}{F(x)}$,
for each $b^*\in\Uball^*$. By virtue of well-known properties
of persistence of convexity under such operations on functions
as translation and taking the supremum over an arbitrary index set,
from the convexity of each function $x\mapsto
\supf{b^*}{F(x)}$ one deduces the convexity of $\charfun{F}{C}$.

(iii) This fact is a straightforward consequence of the property
of support functions recalled in Remark \ref{rem:supfprops}(iii) and
the equality $C=\lambda C$, which is valid for every $\lambda>0$ because
$C$ is a cone.
\end{proof}

\begin{lemma}[Continuity of $\charfun{F}{C}$]   \label{lem:charfuncon}
Let $F:\X\rightrightarrows\Y$ be a set-valued mapping between Banach
spaces. Suppose that:

(i) $F(x)\in\BConv(\Y)\backslash\{\varnothing\}$ for every $x\in\X$;

(ii) $F$ is concave on $\X$;

(iii) $F$ is locally bounded around some $x_0\in\X$, i.e. there
exist constants $\delta,\, \kappa>0$ such that
$$
  F(x)\subseteq\kappa\Uball,\quad\forall x\in\ball{x_0}{\delta}.
$$
Then, function $\charfun{F}{C}$ is continuous on $\X$.
\end{lemma}

\begin{proof}
According to assertions (i) and (ii) in Lemma \ref{lem:charfunprops},
under the above assumptions the function
$\charfun{F}{C}$ is a convex function with $\dom\charfun{F}{C}=\X$.
Notice that, by virtue of hypothesis (iii), $\charfun{F}{C}$ turns
out to be bounded from above on a neighbourhood of $x_0$. Indeed,
by taking into account Remark \ref{rem:supfprops}(iii), one has
$$
  \supf{b^*}{F(x)}\le\supf{b^*}{\kappa\Uball}\le\kappa,
  \quad\forall b^*\in\Uball^*,\ \forall x\in\ball{x_0}{\delta},
$$
and hence
$$
  \charfun{F}{C}(x)\le\sup_{b^*\in\Uball^*}\supf{b^*}{F(x)}-
  \inf_{b^*\in\Uball^*}\supf{b^*}{C}\le\kappa+\|c_0\|,
  \quad\forall x\in\ball{x_0}{\delta},
$$
with $c_0\in C$.
It is a well-known fact in convex analysis that the boundedness
of a convex function on a neighbourhood of a point in its domain
implies the continuity of the function in the interior of its
whole domain (see, for instance, \cite[Theorem 2.2.9]{Zali02}).
Thus, one deduces that $\charfun{F}{C}$ is continuous on $\inte
(\dom\charfun{F}{C})=\X$.
\end{proof}

\begin{remark}[Continuity of $\dcharfun{F}{C}$]
As $\Uball^*\cap\dcone{C}\subseteq\Uball^*$ and $\nullv^*\in
\Uball^*\cap\dcone{C}$, it is not difficult to check that all the
assertions in Lemma \ref{lem:charfunprops} and Lemma \ref{lem:charfuncon}
remain true if replacing $\charfun{F}{C}$ with $\dcharfun{F}{C}$.
\end{remark}

\begin{lemma}     \label{lem:weakstarcont}
Let $F:\X\rightrightarrows\Y$ be a set-valued mapping between Banach
spaces. Suppose that:

(i) $F(x)\in\BConv(\Y)\backslash\{\varnothing\}$ for every $x\in\X$;

(ii) $C\in\BConv(\Y)\backslash\{\varnothing\}$.

\noindent Then, for every $x\in\X$, the function $y^*\mapsto\supf{y^*}{F(x)}-
\supf{y^*}{C}$ is continuous on $\Y^*$ with respect to the \weakstar
topology. If hypothesis $(ii)$ is replaced by

($\dcone{ii}$) $C$ is a closed convex cone,

\noindent then the function $y^*\mapsto\supf{y^*}{F(x)}-
\supf{y^*}{C}$ is continuous on $\dcone{C}$ with respect to the
topology induced by the \weakstar topology.
\end{lemma}

\begin{proof}
Fix an arbitrary $x\in\X$. Since $F(x),\, C\in\BConv(\Y)\backslash
\{\varnothing\}$, then by taking into account what noted in
Remark \ref{rem:supfprops}(iv), one can say that $\supf{\cdot}{F(x)}$
and $\supf{\cdot}{C}$ are sublinear continuous functions on $\Y^*$.
Furthermore, as a convex function, they turn out to be continuous
also with respect to the \weakstar topology on $\Y^*$. Therefore, so
is their difference.

Now, if hypothesis $(ii)$ is replaced by ($\dcone{ii}$), then
one readily sees that
$$
  \supf{y^*}{C}=0,\quad\forall y^*\in\dcone{C},
$$
and hence one obtains
$$
  \supf{y^*}{F(x)}-\supf{y^*}{C}=\supf{y^*}{F(x)},\quad
  \forall y^*\in\dcone{C}.
$$
Since the function $y^*\mapsto\supf{y^*}{F(x)}$ is continuous
with respect to the \weakstar topology on $\Y^*$, the thesis follows
at once.
\end{proof}


\subsection{Variational analysis tools}

Given a nonempty subset $S\subseteq\X$ of a Banach space and
$\bar x\in S$, recall that the contingent cone to $S$ at $\bar x$ is
defined as being
$$
  \Tang{S}{\bar x}=\{v\in\X:\ \exists (v_n)_n,\ v_n\to v,\
  \exists (t_n)_n,\ t_n\downarrow 0:\ \bar x+t_nv_n\in S,
  \ \forall n\in\N\}.
$$
It provides a first-order approximation of $S$ near $\bar x$ and,
as such, it is useful to glean information on the local geometry
of $S$. Some known facts concerning the contingent cone, which
will be exploited in what follows, are listed in the next remark.

\begin{remark}    \label{rem:tangcone}
(i) The contingent cone  to a set $S$ at each of its points is
always a closed cone (and hence, nonempty). It is also convex,
whenever $S$ is so.

(ii) Given arbitrary $S\subseteq\X$ and $\bar x\in S$, the following
functional characterization of $\Tang{S}{\bar x}$ is known to hold
true
$$
   \Tang{S}{\bar x}=\left\{v\in\X:\ \liminf_{t\downarrow 0}
   {\dist{\bar x+tv}{S}\over t}=0\right\}
$$
(see, for instance, \cite[Proposition 11.1.5]{Schi07}).
\end{remark}

After \cite{Ioff00},
a basic variational analysis tool which revealed to be effective in
studying solvability and error bounds is the strong slope: given
a function $\varphi:X\longrightarrow\R\cup\{\mp\infty\}$ defined
on a metric space $(X,d)$ and an element $x_0\in\dom\varphi$, the
strong slope of $\varphi$ at $x_0$ is defined as being:
\begin{eqnarray*}
  \stsl{\varphi}(x_0)=\left\{\begin{array}{ll}
  0, & \hbox{ if $x_0$ is a local minimizer of $\varphi$}, \\
  \limsup_{x\to x_0}{\varphi(x_0)-\varphi(x)\over d(x,x_0)} &
  \hbox{ otherwise.}
  \end{array}
  \right.
\end{eqnarray*}
The following proposition (for its proof, see \cite[Theorem 2.8]{AzeCor06})
and the subsequent remark explain the role of the strong slope
behind the present approach.

\begin{proposition}    \label{pro:AzeCor06}
Let $(X,d)$ be a complete metric space and let $\varphi:X
\longrightarrow\R$ be a continuous function. Assume that
$[\varphi>0]\ne\varnothing$ and that
$$
  \tau=\inf_{x\in[\varphi>0]}\stsl{\varphi}(x)>0.
$$
Then, it is $[\varphi\le 0]\ne\varnothing$ and
$$
  \dist{x}{[\varphi\le 0]}\le {\varphi(x)\over\tau},
  \quad\forall x\in [\varphi>0].
$$
\end{proposition}

\begin{remark}    \label{rem:stslconvchar}
If $\varphi:\X\longrightarrow\R$ is a continuous
convex function on a Banach space, then its strong slope
at a given point can be expressed in terms of the so-called
subdifferential slope. In other words, it holds
$$
  \stsl{\varphi}(x)=\dist{\nullv^*}{\partial\varphi(x)}=
  \inf\{\|x^*\|:\ x^*\in\partial\varphi(x)\},
$$
(see, for instance, \cite[Theorem 5]{FaHeKrOu10}).
\end{remark}

Following \cite{Ioff81},
the next tool of analysis enables one to perform first-order
approximations of set-valued mappings. In contrast with other
possible approaches to the differentiation of multi-valued mappings, which
are based on the local behaviour of a multifunction near a
given point of its graph (such as graphical differentiation,
coderivative calculus, and so on \cite{AubFra09,Mord06}), the
below notion takes under consideration the whole
image through a set-valued mapping of a reference element in its
domain. For this reason, it seems to be more appropriate for the
problem at the issue. Examples and discussions of several topics
in the prederivative theory, included their role in variational
analysis, can be found in \cite{GaGeMa16,Ioff81,Pang11}.

\begin{definition}    \label{def:outpreder}
Let $F:\X\rightrightarrows\Y$ be a set-valued mapping between Banach
spaces and let $\bar x\in\X$. A p.h. set-valued mapping $H:\X
\rightrightarrows\Y$  is said to be an {\it outer prederivative}
of $F$ at $\bar x$ if for every $\epsilon>0$  there exists $\delta>0$
such that
$$
  F(x)\subseteq F(\bar x)+H(x-\bar x)+\epsilon\|x-\bar x\|\Uball,
  \quad\forall x\in\ball{\bar x}{\delta}.
$$
\end{definition}

From Definition \ref{def:outpreder} it is clear that outer prederivatives
are not uniquely defined. In particular, whenever $H$ happens to be an
outer prederivative of $F$ at $\bar x$, any p.h. set-valued mapping
$\tilde H:\X\rightrightarrows\Y$ such that $\tilde H(x)\supseteq H(x)$,
for every $x\in\X$, is still an outer prederivative of $F$ at $\bar x$.

Another clear fact is that any fan admits itself as an outer
prederivative at $\nullv$.


\section{Solution analysis: qualitative and quantitative results}   \label{Sect:3}

\begin{proposition}[Functional characterization of solutions]
\label{pro:funchract}
Given a set-inclusive generalized equation $(\IGE)$, suppose that
$F(x)\in\Conv(\Y)\backslash\{\varnothing\}$ for every $x\in\X$.
It holds
$$
  \Solv{\IGE}=S\cap[\charfun{F}{C}\le 0]=S\cap{\charfun{F}{C}}^{-1}(0).
$$
If, in particular, the convex set $C$ is a cone, then it holds
$$
  \Solv{\IGE}=S\cap[\dcharfun{F}{C}\le 0]=S\cap{\dcharfun{F}{C}}^{-1}(0).
$$
\end{proposition}

\begin{proof}
If $x\in\Solv{\IGE}$, then $x\in S$ and $F(x)\subseteq C$. As recalled
in Remark \ref{rem:supfprops}(iii), this inclusion implies $\supf{b^*}{F(x)}
\le\supf{b^*}{C}$ for every $b^*\in\Uball^*$. On account of the definition
of $\charfun{F}{C}$, the last inequality leads clearly to $\charfun{F}{C}
(x)\le 0$.

Conversely, if $x\in S\cap [\charfun{F}{C}\le 0]$, then according to the
definition of $\charfun{F}{C}$, one has
$$
   \supf{b^*}{F(x)}-\supf{b^*}{C}\le 0,\quad\forall b^*\in\Uball^*.
$$
Since $F(x),\, C\in\Conv(\Y)$, by virtue of what observed in Remark
\ref{rem:supfprops}(iii), the last inequality suffices to deduce
that $F(x)\subseteq C$, so $x\in\Solv{\IGE}$.

As for the second assertion,
since $\dcharfun{F}{C}(x)\le\charfun{F}{C}(x)$ for every $x\in\X$,
if $x\in\Solv{\IGE}$ then, as a consequence of what has been proved
above, one can state that $S\cap[\dcharfun{F}{C}\le 0]$.

Now, suppose that
\begin{equation}     \label{in:sufpartchar}
  \sup_{b^*\in\Uball^*\cap\dcone{C}}[\supf{b^*}{F(x)}-\supf{b^*}{C}]\le 0.
\end{equation}   \label{in:sufpartcharsol}
Ab absurdo assume that $F(x)\not\subseteq C$, that is there exists
$y_0\in F(x)$ such that $y_0\not\in\ C$. By the strict separation
theorem (see, for instance, \cite[Theorem 1.1.5]{Zali02}) there
exist $y^*\in\Y^*\backslash\{\nullv^*\}$ and $\alpha\in\R$ such that
\begin{equation}    \label{in:strsepar}
  \langle y^*,y_0\rangle>\alpha>\langle y^*,y\rangle,
  \quad\forall y\in C.
\end{equation}
Notice that it must be $\alpha>0$ inasmuch $C$, as a closed convex
cone, contains $\nullv$. As a consequence, one can deduce that
$y^*\in\dcone{C}$. Indeed, if there were $c_0\in C\backslash\{
\nullv\}$ such that $\langle y^*,c_0\rangle>0$, one would have
$\lambda c_0\in C$ also for $\lambda>{\alpha\over\langle y^*,
c_0\rangle}>0$, so that
$$
  \langle y^*,\lambda c_0\rangle=\lambda \langle y^*,c_0\rangle>
  \alpha,
$$
which contradicts the second inequality in $(\ref{in:strsepar})$. Thus, by
defining $b^*_0=y^*/\|y^*\|\in\Uball^*\cap\dcone{C}$,
one finds
$$
  \supf{b^*_0}{F(x)}-\supf{b^*_0}{C}>{\alpha\over\|y^*\|}>0.
$$
The last chain of inequalities is inconsistent with inequality
$(\ref{in:sufpartchar})$.
To conclude, observe that $[\charfun{F}{C}\le 0]={\charfun{F}{C}}^{-1}
(0)$ and $[\dcharfun{F}{C}\le 0]={\dcharfun{F}{C}}^{-1}(0)$ because
$F$ takes nonempty values, so $\charfun{F}{C}$ and $\dcharfun{F}{C}$
are nonnegative functions. This completes the proof.
\end{proof}

The reader should notice that the main effect of Proposition
\ref{pro:funchract} in studying a problem $(\IGE)$ is to allows
one to reformulate it in variational terms: solutions to $(\IGE)$
become not only zeros but also global minimizers for the functions
$\charfun{F}{C}$ and $\dcharfun{F}{C}$. Such a reformulation paves
the way to many analysis approaches currently at disposal in convex
optimization.

The next proposition takes profit from the above characterization
in order to single out general qualitative properties of
$\Solv{\IGE}$.

\begin{proposition}[Closure and convexity of $\Solv{\IGE}$]
\label{pro:cloconvsol}
Let a set-inclusive generalized equation $(\IGE)$ be given.
Under the hypotheses (i)--(iii) of Lemma \ref{lem:charfuncon}
$\Solv{\IGE}$ is a (possibly empty) closed and convex set.
\end{proposition}

\begin{proof}
In the light of Proposition \ref{pro:funchract}, the
thesis is a straightforward consequence of the fact that,
upon the assumptions made, functions $\charfun{F}{C}$ and
$\dcharfun{F}{C}$ are convex and continuous functions.
\end{proof}

Henceforth, in order to concentrate on the role of the data $F$ and $C$,
it will be assumed $S=\X$.

A reasonable question related to a problem $(\IGE)$ one may pose is the
solution existence. Within the present variational approach,
a condition can be formulated by means of the following
infinitesimal constructions.
Given a generalized equation of the form $(\IGE)$ and $x\in\X$,
if $C\in\BConv(\Y)\backslash\{\varnothing\}$, let us define
$$
  B_x=\{b^*\in\Uball^*:\ \supf{b^*}{F(x)}-\supf{b^*}{C}\}=
  \charfun{F}{C}(x)\},
$$
and
\begin{equation}    \label{eq:erbocdef}
   \erboc{F}{x}=\inf\left\{\|x^*\|:\ x^*\in\wclco\left(\bigcup_{b^*\in B_x}
   \partial\supf{b^*}{F(\cdot)}(x)\right)\right\}.
\end{equation}
Analogously, in the case in which $C$ is a closed convex cone,
let us define
$$
  \dcone{B_x}=\{b^*\in\Uball^*\cap\dcone{C}:\ \supf{b^*}{F(x)}-\supf{b^*}{C}
  =\dcharfun{F}{C}(x)\},
$$
and
$$
   \derboc{F}{x}=\inf\left\{\|x^*\|:\ x^*\in\wclco\left(\bigcup_{b^*\in\dcone{B_x}}
   \partial\supf{b^*}{F(\cdot)}(x)\right)\right\}.
$$
The quantity $\erboc{F}{x}$ (resp. $\derboc{F}{x}$) can be interpreted as
a set-valued counterpart for the concept of slope of a functional. Therefore,
one naturally expects that its behaviour affects the existence of minimizers
of $\charfun{F}{C}$ (resp. $\dcharfun{F}{C}$), and hence of solutions to $(\IGE)$.
Notice that, fixed any $x\in\X$, since the function $y^*\mapsto\supf{y^*}{F(x)}
-\supf{y^*}{C}$ is continuous with respect to the \weakstar topology on the
\weakstar compact set $\Uball^*$ (recall Lemma \ref{lem:weakstarcont}),
then $B_x\ne\varnothing$. Thus, under the hypotheses of Lemma \ref{lem:weakstarcont}
the quantity $\erboc{F}{x}$ is finite. The same, of course, is true for
$\derboc{F}{x}$.

\begin{remark}    \label{rem:argmaxsphere}
In view of further considerations, it is useful to note that, for every
$x\in [\charfun{F}{C}>0]$, it must be  $B_x\subseteq\Usfer^*$. Indeed,
according to Proposition \ref{pro:funchract}, since it is $\supf{\nullv^*}{F(x)}
-\supf{\nullv^*}{C}=0$, one has $\nullv^*\not\in B_x$. Moreover,
since function $b^*\mapsto \supf{b^*}{F(x)}-\supf{b^*}{C}$ is p.h. (actually,
difference of sublinear functions) on $\X^*$, if it were
$b^*\in B_x$ with $\|b^*\|<1$, one would reach the absurdum
\begin{eqnarray*}
   \supf{b^*/\|b^*\|}{F(x)}-\supf{b^*/\|b^*\|}{C} &=& {1\over \|b^*\|}
   \left(\supf{b^*}{F(x)}-\supf{b^*}{C}\right)   \\
   &>& \supf{b^*}{F(x)}-\supf{b^*}{C}=\charfun{F}{C}(x),
\end{eqnarray*}
while it is $b^*/\|b^*\|\in\Uball^*$.
\end{remark}

The next result provides a sufficient condition for the solvability
of a problem $(\IGE)$, complemented with a global estimate of the
distance from its solution set (error bound). As such, it provides
qualitative and quantitative information on $\Solv{\IGE}$.

\begin{theorem}[Solvability and global error bound]    \label{thm:solerbo}
With reference to a generalized equation of the form $(\IGE)$, suppose
that:

(i) $F(x)\in\BConv(\Y)\backslash\{\varnothing\}$ for every $x\in\X$;

(ii) $F$ is concave on $\X$;

(iii) $F$ is locally bounded around some $x_0\in\X$;

(iv) $C\in\BConv(\Y)\backslash\{\varnothing\}$ and
$$
  \erbom{F}=\inf\{\erboc{F}{x}:\ x\in [\charfun{F}{C}>0]\}>0.
$$
\noindent Then, $\Solv{\IGE}\ne\varnothing$ and it holds
\begin{equation}    \label{in:erbo1}
  \dist{x}{\Solv{\IGE}}\le{\charfun{F}{C}(x)\over \erbom{F}},
  \quad\forall x\in\X.
\end{equation}
If hypothesis $(iv)$ is replaced by

($\dcone{iv}$) $C$ is a closed convex cone and
$$
  \derbom{F}=\inf\{\derboc{F}{x}:\ x\in [\dcharfun{F}{C}>0]\}>0,
$$
\noindent then, $\Solv{\IGE}\ne\varnothing$ and it holds
\begin{equation}    \label{in:erbo2}
  \dist{x}{\Solv{\IGE}}\le{\dcharfun{F}{C}(x)\over \derbom{F}},
  \quad\forall x\in\X.
\end{equation}
\end{theorem}

\begin{proof}
In the light of Proposition \ref{pro:funchract}, the proof consists
in checking that, under the assumptions made, it is possible to apply
Proposition \ref{pro:AzeCor06}, with $X=\X$ and $\varphi\in\{\charfun{F}{C},\,
\dcharfun{F}{C}\}$.

Let us start with noting that, owing to hypotheses (i)--(iii), one can
invoke Lemma \ref{lem:charfunprops}(i) and (ii) as well as Lemma \ref{lem:charfuncon}.
So $\charfun{F}{C}$ is a continuous convex function on $\X$.
Notice that, if $[\charfun{F}{C}>0]=\varnothing$ or $[\dcharfun{F}{C}>0]=
\varnothing$, then on account of Proposition \ref{pro:funchract} it is
$\Solv{\IGE}=\X$, so all assertions in the thesis trivially follow.
Therefore, one can assume that $[\charfun{F}{C}>0]\ne\varnothing$ or
$[\dcharfun{F}{C}>0]\ne\varnothing$ (depending on the assumption on $C$
is being made).

Now, in the case in which hypothesis (iv) holds true, as $\charfun{F}{C}$
is a continuous convex function on $\X$, according to Remark \ref{rem:stslconvchar}
one has
\begin{equation}     \label{eq:stslsubdifrep}
  \stsl{\charfun{F}{C}}(x)=\dist{\nullv^*}{\partial\charfun{F}{C}(x)}=
  \inf\{\|x^*\|:\ x^*\in\partial\charfun{F}{C}(x)\}.
\end{equation}
Notice that, by Lemma \ref{lem:weakstarcont}, for each $x\in\X$ the function
$y^*\mapsto\supf{y^*}{F(x)}-\supf{y^*}{C}$ is continuous on the \weakstar
compact set $\Uball^*$, with respect to the \weakstar topology. By taking
into account what recalled in Remark \ref{rem:DubMil}, with $\Xi=\Uball^*$,
one obtains
\begin{eqnarray}   \label{eq:charfunsubd}
   \partial\charfun{F}{C}(x) &=& \wclco\left(\bigcup_{b^*\in B_x}
   \partial(\supf{b^*}{F(\cdot)}-\supf{b^*}{C})(x)\right)  \nonumber \\
   &=& \wclco\left(\bigcup_{b^*\in B_x}
   \partial\supf{b^*}{F(\cdot)}(x)\right).
\end{eqnarray}
Thus, by recalling formulae $(\ref{eq:erbocdef})$ and $(\ref{eq:stslsubdifrep})$,
one finds
$$
  \inf_{x\in[\charfun{F}{C}>0]}\stsl{\charfun{F}{C}}(x)=\erbom{F}>0.
$$
This makes it possible to employ Proposition \ref{pro:AzeCor06}, whence
the first part of the assertion follows at once.

In the case in which hypothesis (iv) is replaced by $(\dcone{iv})$, each
function $y^*\mapsto\supf{y^*}{F(x)}-\supf{y^*}{C}$ turns out to be continuous
with respect to the \weakstar topology on the \weakstar compact space $\Uball^*
\cap\dcone{C}$. It remains to adapt equalities in $(\ref{eq:charfunsubd})$
to the current case, by taking into account the definition of $\derboc{F}{x}$.
This completes the proof.
\end{proof}

\begin{remark}    \label{rem:regdist0}
(i) As a first comment to Theorem \ref{thm:solerbo}, it is worth noting that
the condition $\erbom{F}>0$ (resp. $\derbom{F}>0$) translates in terms of
problem data the well-known condition $\nullv^*\not\in\partial\charfun{F}{C}
(x)$ (resp. $\nullv^*\not\in\partial\dcharfun{F}{C}(x)$) for the validity
of a global error bound in the convex setting (see, for instance
\cite[Theorem 5]{FaHeKrOu10}).

(ii) As it happens in general for global error bounds, one can observe
that inequality $(\ref{in:erbo1})$ qualifies $\Solv{\IGE}$ as a set of
weak sharp minimizers of $\charfun{F}{C}$. Recall that a closed set $S\subseteq\X$
is said to be a set of weak sharp minimizers if there exists $\alpha>0$
such that
$$
   \varphi(x)\ge\inf_{x\in\X}\varphi(x)+\alpha\dist{x}{S},\quad\forall
   x\in\X
$$
(see, for instance, \cite[Section 3.10]{Zali02}).
Such a property entails the fact that for any minimizing sequence $(x_n)_n$,
i.e. any sequence in $\X$ such that $\charfun{F}{C}(x_n)\to 0$ as
$n\to\infty$, one has that $\dist{x_n}{\Solv{\IGE}}\to 0$, that is a kind of
generalization of the Tikhonov well-posedness. In other words, it
prescribes a certain variational behaviour, with which the minimum of $\charfun{F}{C}$
is attained.
\end{remark}

Error bounds are not only interesting in themselves, but trigger
several facts which help to better understand the geometry of
$\Solv{IGE}$, a set often difficult to be determined explicitly.
According to a widely used scheme of analysis, they may
be exploited to provide approximated representations of the solution
set to $(\IGE)$. This is done in the next theorem by employing the
notion of contingent cone.

\begin{theorem}[Tangential characterization of $\Solv{\IGE}$]    \label{thm:soltangchar}
With reference to a set-inclusive generalized equation $(\IGE)$,
suppose that:

(i) $\bar x\in\Solv{IGE}$;

(ii) $F(x)\in\BConv(\Y)\backslash\{\varnothing\}$ for every $x\in\X$;

(iii) $F$ is concave on $\X$;

(iv) $F$ is locally bounded around some $x_0\in\X$;

(v) $C\in\BConv(\Y)\backslash\{\varnothing\}$ and $\erbom{F}>0$.

\noindent Then, it results in
$$
  \Tang{\Solv{\IGE}}{\bar x}=[\charfun{F}{C}'(\bar x;\cdot)\le 0].
$$
If hypothesis $(v)$ is replaced by

($\dcone{v}$) $C$ is a closed convex cone and $\derbom{F}>0$,

\noindent then it results in
$$
  \Tang{\Solv{\IGE}}{\bar x}=[(\dcharfun{F}{C})'(\bar x;\cdot)\le 0].
$$
\end{theorem}

\begin{proof}
Let us start with supposing that hypotheses (i)--(v) are in force.
In such a circumstance, as already seen, $\charfun{F}{C}$ is a
convex continuous function on $\X$ and, according to Proposition
\ref{pro:cloconvsol} and Theorem \ref{thm:solerbo}, $\Solv{\IGE}$
is a nonempty, closed and convex set. Consequently, $\charfun{F}{C}'
(\bar x;\cdot)$ is sublinear and Lipschitz continuous on $\X$
and the function $x\mapsto\dist{x}{\Solv{\IGE}}$ is Lipschitz
continuous and convex on $\X$.

To show that $\Tang{\Solv{\IGE}}{\bar x}\supseteq[\charfun{F}{C}'
(\bar x;\cdot)\le 0]$, take an arbitrary $v\in [\charfun{F}{C}'
(\bar x;\cdot)\le 0]$. Since the error bound estimate in $(\ref{in:erbo1})$
is valid, one can write
$$
  \liminf_{t\downarrow 0}{\dist{\bar x+tv}{\Solv{\IGE}}
  \over t}\le
  \liminf_{t\downarrow 0}{\charfun{F}{C}(\bar x+tv)\over
  \erbom{F} t}={\charfun{F}{C}'(\bar x;v)\over \erbom{F}}\le 0.
$$
Thus, by virtue of the characterization recalled in Remark \ref{rem:tangcone}
(ii), from the last inequality the inclusion $v\in \Tang{\Solv{\IGE}}{\bar x}$
immediately follows.

In order to prove the reverse inclusion, take an arbitrary $v\in
\Tang{\Solv{\IGE}}{\bar x}$. This means that there exists a
sequence $(t_n)_n$, with $t_n\downarrow 0$, such that
\begin{equation}    \label{eq:limdist}
   \lim_{n\to\infty} {\dist{\bar x+t_nv}{\Solv{\IGE}}\over t_n}=0.
\end{equation}
Since $\charfun{F}{C}$, as a continuous convex function, is also
locally Lipschitz around $\bar x$ (see \cite[Corollary 2.2.13]{Zali02}),
there exist real $\kappa,\, r>0$ such that
\begin{eqnarray}    \label{in:loclipcharfun}
    \charfun{F}{C}(x) &=& |\charfun{F}{C}(x)-\charfun{F}{C}(z)|
    \le\kappa\|x-z\|,    \\
    & &\qquad\qquad\quad\forall x\in\ball{\bar x}{r},\
    \forall z\in\ball{\bar x}{r}\cap\Solv{\IGE}. \nonumber
\end{eqnarray}
Now, it is proper to observe that
$$
  \dist{x}{\Solv{\IGE}}=\dist{x}{\Solv{\IGE}\cap\ball{\bar x}{r}},
  \quad\forall x\in\ball{\bar x}{r/2}.
$$
From the last equality, by taking into account inequality $(\ref{in:loclipcharfun})$,
one obtains
\begin{eqnarray}    \label{in:charfunledist}
  \charfun{F}{C}(x) &\le &\kappa\inf_{z\in\ball{\bar x}{r}\cap\Solv{\IGE}}\|x-z\|
   \\
  &=& \kappa\dist{x}{\Solv{\IGE}},\qquad\forall x\in\ball{\bar x}{r/2}.  \nonumber
\end{eqnarray}
By combining $(\ref{eq:limdist})$ with $(\ref{in:charfunledist})$,
one obtains
$$
  \lim_{n\to\infty} {\charfun{F}{C}(\bar x+t_nv)\over t_n}=0.
$$
This means that $v\in[\charfun{F}{C}'(\bar x;\cdot)\le 0]$, thereby
proving the first assertion in the thesis.

The second assertion can be proved in a similar manner, by making use
of the error bound estimate in $(\ref{in:erbo2})$.
\end{proof}

Another topic that can be developed as a consequence of error bounds
are penalty methods. In the present context, this can be done
for optimization problems, whose feasible region is
defined by a constraint system formalized as a $(\IGE)$ problem, i.e.
$$
  \min\vartheta(x) \qquad\hbox{ subject to }\qquad F(x)\subseteq C.
  \leqno (\mathcal{P})
$$
Given a solution $\bar x\in\Solv{\IGE}$ to $(\mathcal{P})$,
under a Lipschitz assumption on $\vartheta$ and the validity of
Theorem \ref{thm:solerbo}, it is possible to prove the existence
of a penalty parameter $\lambda>0$ such that $\bar x$ is also
solution to the unconstrained optimization problem
$$
  \min_{x\in\X}[\vartheta(x)+\lambda\charfun{F}{C}(x)],
$$
that is an exact penalization holds. Since this kind of result
can be proved by standard arguments (see, for instance, \cite[Theorem 5.2]{Uder19}),
the details are omitted here. What is more important to note is
that, whenever an exact penalization takes place, one can
develop optimality conditions for $(\mathcal{P})$, by exploiting the subdifferential
calculus rules, starting from the conditions valid for unconstrained
problems.

\vskip1cm


\section{Estimates via prederivatives}    \label{Sect:4}

The findings of the preceding section are expressed in terms of problem
data through the function $\charfun{F}{C}$. It comes natural to
investigate how the basic condition for solvability and error
bound, namely the positivity of the constant $\erbom{F}$, can be guaranteed in the case
the mapping $F$ is assumed to be locally approximated by another
set-valued mapping $H$, with a simpler structure. In what follows
this is done by employing outer prederivatives as a first-order
approximation of $F$ at a reference point.
To this aim, with a given p.h. set-valued mapping $H:\X\rightrightarrows\Y$,
let us associate the function $\charfan{H}:\X\longrightarrow\R\cup\{\mp\infty\}$,
defined by
$$
  \charfan{H}(v)=\sup_{b^*\in\Usfer^*}\supf{b^*}{H(v)},
  \quad\forall v\in\X.
$$
Notice that, by Lemma \ref{lem:charfunprops}, if $H$ takes
nonempty closed, bounded and convex values for every $x\in\X$, $\charfan{H}$
is p.h. and real-valued.

\begin{proposition}     \label{pro:preapprox}
Let $F:\X\rightrightarrows\Y$ be a set-valued mapping between Banach
spaces and let $x_0\in [\charfun{F}{C}>0]$. Suppose that:

(i) $F(x)\in\BConv(\Y)\backslash\{\varnothing\}$ for every $x\in\X$;

(ii) $F$ is concave on $\X$;

(iii) $F$ is locally bounded around some element of $\X$;

(iv) $C\in\BConv(\Y)\backslash\{\varnothing\}$;

(v) $F$ admits an outer prederivative $H:\X\rightrightarrows\Y$
at $x_0$ such that
$$
   H(\nullv)=\{\nullv\} \quad\hbox{ and }\quad
   H(x)\in\BConv(\Y)\backslash\{\varnothing\},\quad\forall x\in\X.
$$
Then, it holds
$$
  \charfun{F}{C}'(x_0;v)\le\charfan{H}(v),\quad\forall v\in\X
$$
and, consequently,
\begin{equation}     \label{in:subdifapproxhom}
    \partial\charfun{F}{C}(x_0)\subseteq\partial\charfan{H}(\nullv).
\end{equation}
\end{proposition}

\begin{proof}
Observe first that, since under the hypotheses (i)-(iii) the function
$\charfun{F}{C}$ is continuous on $\X$ by Lemma \ref{lem:charfuncon},
then the set $[\charfun{F}{C}>0]$ is open. Consequently, there
exists $\delta_0>0$ such that $\ball{x_0}{\delta_0}\subseteq[\charfun{F}{C}>0]$.
Since $H$ is an outer prederivative of $F$ at $x_0$, fixed any $\epsilon>0$
there exists $\delta>0$ such that
$$
  F(x)\subseteq F(x_0)+H(x-x_0)+\epsilon\|x-x_0\|\Uball,
  \quad\forall x\in\ball{x_0}{\delta}.
$$
Without loss of generality, it is possible to take $\delta\in (0,\delta_0)$.
On the base of Remark \ref{rem:supfprops}(ii) and (iii), the above
inclusion implies for any $b^*\in\Uball^*$
$$
  \supf{b^*}{F(x)}\le\supf{b^*}{F(x_0)}+\supf{b^*}{H(x-x_0)}+
  \epsilon\|x-x_0\|,\quad\forall x\in\ball{x_0}{\delta},
$$
whence
\begin{eqnarray*}
  \supf{b^*}{F(x)}-\supf{b^*}{C} &\le& \supf{b^*}{F(x_0)}-\supf{b^*}{C}
  +\supf{b^*}{H(x-x_0)}  \\
  &+& \epsilon\|x-x_0\|,  \qquad\forall x\in
  \ball{x_0}{\delta}.
\end{eqnarray*}
By taking the supremum over the set $\Usfer^*$ in both sides of the
last inequality and recalling that $B_x\subseteq\Usfer^*$, as noted
in Remark \ref{rem:argmaxsphere}, provided that $x\in\ball{x_0}{\delta}
\subseteq[\charfun{F}{C}>0]$, one obtains
\begin{eqnarray*}    \label{in:supfhom}
  \charfun{F}{C}(x)&=&\sup_{b^*\in\Usfer^*}[\supf{b^*}{F(x)}-\supf{b^*}{C}] \\
  &\le& \sup_{b^*\in\Usfer^*}\bigl[\supf{b^*}{F(x_0)}-\supf{b^*}{C}
  +\supf{b^*}{H(x-x_0)}\bigl]+\epsilon\|x-x_0\| \\
  &\le& \charfun{F}{C}(x_0)+\charfan{H}(x-x_0)+\epsilon\|x-x_0\|,
  \quad\forall x\in \ball{x_0}{\delta}.
\end{eqnarray*}
Thus, if taking $x=x_0+tv$, with $t\in (0,\delta)$ and $v\in\Uball$,
it is clearly $x\in  \ball{x_0}{\delta}$ so, by the last inequality, it
results in
$$
  {\charfun{F}{C}(x_0+tv)-\charfun{F}{C}(x_0)\over t}\le
  \charfan{H}(v)+\epsilon,\quad\forall t\in (0,\delta),
  \ \forall v\in\Uball.
$$
By passing to the limit as $t\downarrow 0$ in the above inequality,
one finds
$$
  \charfun{F}{C}'(x_0;v)\le\charfan{H}(v)+\epsilon,
  \quad\forall v\in\Uball,
$$
which, by arbitrariness of $\epsilon>0$, gives
$$
  \charfun{F}{C}'(x_0;v)\le\charfan{H}(v),
  \quad\forall v\in\Uball.
$$
As $\charfun{F}{C}'(x_0;\cdot)$ and $\charfan{H}$ are p.h.
functions, the first assertion is the thesis follows.

The second assertion is a straightforward consequence of the first
one, because $\charfan{H}(\nullv)=0$ as it is $H(\nullv)=\{\nullv\}$
and, according to formula $(\ref{eq:MorRoc})$, it is $\charfun{F}{C}'
(x_0;\cdot)=\supf{\cdot}{\partial\charfun{F}{C}(x_0)}$ and
$\charfan{H}=\supf{\cdot}{\partial\charfan{H}(\nullv)}$.
This completes the proof.
\end{proof}

It is noteworthy that, under the hypotheses of Proposition \ref{pro:preapprox},
$\charfun{F}{C}$ is a continuous and convex function, so
$\partial\charfun{F}{C}(x_0)\ne\varnothing$. This entails that, even
though $\charfan{H}$ fails to be sublinear ($H$ being not necessarily
concave), in this circumstance it happens that $\partial\charfan{H}(\nullv)
\ne\varnothing$.

Hereafter, in order to provide verifiable conditions for the validity
of error bounds, $F$ will be assumed to admit special prederivatives,
which can be represented as fans generated by proper families of linear bounded
operators (remember Example \ref{ex:fan}).

In this concern,
given a weakly closed and convex subset $\mathcal{G}\subseteq\Lin(\X,\Y)$,
letting $\mathcal{G}^*=\{\Lambda^*\in\Lin(\Y^*,\X^*):\ \Lambda\in\mathcal{G}\}$,
define
$$
  \mathcal{G}^*(\Usfer^*)=\left\{x^*\in\X^*:\ x^*\in
  \bigcup_{\Lambda^*\in\mathcal{G}^*}\Lambda^*\Usfer^*
  \right\}
$$
and
$$
  \Bconst{\mathcal{G}^*}=\sup_{v\in\Uball}\inf_{x^*\in
  \mathcal{G}^*(\Usfer^*)}\langle x^*,v\rangle.
$$
Notice that, whereas for every $v\in\Uball$ it is
$\inf_{x^*\in\mathcal{G}^*(\Uball^*)}\langle x^*,v\rangle\le 0$
because $\nullv^*\in\mathcal{G}^*(\Uball^*)$, and hence
$\sup_{v\in\Uball}\inf_{x^*\in\mathcal{G}^*(\Uball^*)}\langle
x^*,v\rangle\le 0$, it may actually happen that $\Bconst{\mathcal{G}^*}
>0$, for a given $\mathcal{G}\subseteq\Lin(\X,\Y)$. The next
propositions show that such an event is a favourable circumstance
for the validity of an error bound.

\begin{proposition}    \label{pro:distnullv*}
Let $F:\X\rightrightarrows\Y$ be a set-valued mapping between Banach
spaces and let $x_0\in [\charfun{F}{C}>0]$.
Under the hypotheses of Proposition \ref{pro:preapprox}, suppose that
$H:\X\rightrightarrows\Y$ is a fan generated by a (nonempty) weakly closed,
bounded and convex set $\mathcal{G}_{x_0}\subseteq\Lin(\X,\Y)$,
satisfying the condition
\begin{equation}    \label{in:fanreg}
   \Bconst{\mathcal{G}^*_{x_0}}>0.
\end{equation}
Then, it holds
$$
  \Bconst{\mathcal{G}^*_{x_0}}\le
  \dist{\nullv^*}{\partial\charfun{F}{C}(x_0)}.
$$
In particular, it is
\begin{equation*}
    \nullv^*\not\in\partial\charfun{F}{C}(x_0).
\end{equation*}

\end{proposition}

\begin{proof}
Under the assumptions made, as it is $H(\nullv)=\{\Lambda\nullv:\
\Lambda\in\mathcal{G}_{x_0}\}=\{\nullv\}$ and
$H(x)\in\BConv(\Y)\backslash\{\varnothing\}$ for every $x\in\X$,
it is possible to apply Proposition \ref{pro:preapprox}, in such
a way to get inclusion $(\ref{in:subdifapproxhom})$.
Consequently, by taking into account the estimate recalled in Remark
\ref{rem:supfprops}(v) along with the representation in formula
$(\ref{eq:MorRoc})$, one has
\begin{eqnarray*}
   \dist{\nullv^*}{\partial\charfun{F}{C}(x_0)}&\ge&
   \dist{\nullv^*}{\partial\charfan{H}(\nullv)}\ge
   -\inf_{v\in\Uball}\supf{v}{\partial\charfan{H}(\nullv)}  \\
   &=& -\inf_{v\in\Uball}\charfan{H}(v).
\end{eqnarray*}
By recalling the definition of $\charfan{H}$ and of $H$, one obtains
\begin{eqnarray*}
   -\inf_{v\in\Uball}\charfan{H}(v) &=&-\inf_{v\in\Uball}\sup_{b^*\in\Usfer^*}
   \supf{b^*}{H(v)}=\sup_{v\in\Uball}\inf_{b^*\in\Usfer^*}\inf_{\Lambda\in\mathcal{G}_{x_0}}
   \inf\langle b^*,\Lambda(-v)\rangle \\
   &=&\sup_{v\in\Uball}\inf_{b^*\in\Usfer^*}\inf_{\Lambda^*\in\mathcal{G}^*_{x_0}}
   \langle\Lambda^*b^*,-v\rangle=\Bconst{\mathcal{G}^*_{x_0}}>0.
\end{eqnarray*}
The second assertion in the thesis comes as an obvious consequence
of the first one.
\end{proof}

In order to establish a solvability and global error bound result,
the condition formulated in $(\ref{in:fanreg})$ must be satisfied
all over $[\charfun{F}{C}>0]$. Such a requirement naturally leads
to introduce the following quantity
$$
  \flat_F=\inf_{x\in[\charfun{F}{C}>0]}\Bconst{\mathcal{G}^*_x}.
$$
Besides, given $\mathcal{G}\subseteq\Lin(\X,\Y)$, and hence
$\mathcal{G}^*=\{\Lambda^*\in\Lin(\Y^*,\X^*):\ \Lambda\in\mathcal{G}\}$,
define
$$
  \tilde{\mathcal{G}_x^*}(\Uball^*)=\left\{x^*\in\X^*:\ x^*\in
  \bigcup_{\Lambda^*\in\mathcal{G}^*}\Lambda^*\Uball^*,
  \right\}.
$$

\begin{corollary}      \label{cor:solverbofan}
With reference to a generalized equation of the form $(\IGE)$, suppose
that:

(i) $F(x)\in\BConv(\Y)\backslash\{\varnothing\}$ for every $x\in\X$;

(ii) $F$ is concave on $\X$;

(iii) $F$ is locally bounded around some element of $\X$;

(iv) $C\in\BConv(\Y)\backslash\{\varnothing\}$;

(v) $F$ admits at each point $x\in\X$ an outer prederivative, which
is a fan generated by a weakly closed, bounded and convex set
$\mathcal{G}_x\subseteq\Lin(\X,\Y)$;

(vi) it holds $\flat_F>0$.

\noindent Then, $\Solv{\IGE}\ne\varnothing$ and it holds
$$
  \dist{x}{\Solv{\IGE}}\le{\charfun{F}{C}(x)\over\flat_F},
  \quad\forall x\in\X.
$$
Moreover, if $\bar x\in\Solv{\IGE}$, it results in
$$
  \Tang{\Solv{\IGE}}{\bar x}\supseteq \bigcap_{x^*\in
  \mathcal{G}_{\bar x}^*(\Uball^*)}[x^*\le 0].
$$
\end{corollary}

\begin{proof}
In the light of Theorem \ref{thm:solerbo} and Proposition \ref{pro:distnullv*},
the first assertion in the thesis follows at once from the inequality chain
\begin{eqnarray*}
   \erbom{F} &=& \inf_{x\in[\charfun{F}{C}>0]}\stsl{\charfun{F}{C}}
   (x)=\inf_{x\in[\charfun{F}{C}>0]}\dist{\nullv^*}{\partial\charfun{F}{C}(x)}
   \\
   &\ge &\inf_{x\in[\charfun{F}{C}>0]}\Bconst{\mathcal{G}^*_x}
   =\flat_F>0.
\end{eqnarray*}
As for the second assertion, fixed $\bar x\in\Solv{\IGE}$, by Theorem \ref{thm:soltangchar}
one has $\Tang{\Solv{\IGE}}{\bar x}=[\charfun{F}{C}'(\bar x;\cdot)\le 0]$.
Since $F$ admits as an outer prederivative at $\bar x$ the fan generated
by $\mathcal{G}_{\bar x}$, by reasoning as in the proof of Proposition
\ref{pro:preapprox} one finds
\begin{eqnarray*}
  \charfun{F}{C}(x)&=&\sup_{b^*\in\Uball^*}[\supf{b^*}{F(x)}-\supf{b^*}{C}] \\
  &\le& \sup_{b^*\in\Uball^*}\bigl[\supf{b^*}{F(\bar x)}-\supf{b^*}{C}
  +\supf{b^*}{H(x-x_0)}\bigl]+\epsilon\|x-x_0\| \\
  &\le& \charfun{F}{C}(\bar x)+\sup_{b^*\in\Uball^*}\supf{b^*}{H(x-x_0)}+
  \epsilon\|x-x_0\|, \quad\forall x\in \ball{x_0}{\delta},
\end{eqnarray*}
for a proper $\delta>0$. This evidently implies
\begin{equation}    \label{in:dirdersupfH}
  \charfun{F}{C}'(\bar x;v)\le \sup_{b^*\in\Uball^*}\supf{b^*}{H(v)},
  \quad\forall v\in\X.
\end{equation}
By making use of the dual representation of $H$, one has
$$
  \sup_{b^*\in\Uball^*}\supf{b^*}{H(v)}=\sup_{b^*\in\Uball^*}
  \sup_{\Lambda^*\in\mathcal{G}_{\bar x}^*}\langle\Lambda^*b^*,
  v\rangle=\sup_{x^*\in\mathcal{G}_{\bar x}^*(\Uball^*)}\langle x^*,
  v\rangle, \quad\forall v\in\X.
$$
Thus, it is $\sup_{b^*\in\Uball^*}\supf{b^*}{H(v)}\le 0$ iff
$$
  v\in [x^*\le 0],\quad\forall x^*\in\mathcal{G}_{\bar x}^*(\Uball^*).
$$
This fact, on account of inequality $(\ref{in:dirdersupfH})$, shows
the validity of the inclusion in the thesis, thereby completing the proof.
\end{proof}

\begin{remark}
From Corollary the proof of \ref{cor:solverbofan} one sees that, at the price of approximating
$F$ with outer prederivatives, the satisfaction  of the basic condition
$\erbom{F}>0$ can be achieved by imposing $\flat_F>0$. As a comment to
the latter condition, it could be relevant to  point out that, fixed
any $x_0\in[\charfun{F}{C}>0]$, whenever the equality
\begin{equation}   \label{eq:minmax}
    \sup_{v\in\Uball}\inf_{x^*\in\mathcal{G}_{x_0}^*(\Usfer^*)}
    \langle x^*,v\rangle=\inf_{x^*\in\mathcal{G}_{x_0}^*(\Usfer^*)}
    \sup_{v\in\Uball}\langle x^*,v\rangle
\end{equation}
holds true, one would be enabled to express $\flat_F>0$ in terms of
Banach constants. Indeed, it is evident that
$$
  \inf_{x^*\in\mathcal{G}_{x_0}^*(\Usfer^*)}\sup_{v\in\Uball}\langle x^*,v\rangle
  =\inf_{\Lambda^*\in\mathcal{G}_{x_0}^*}\inf_{u^*\in\Usfer^*}\|\Lambda^*u^*\|.
$$
The quantity $\dBconst{\Lambda}=\inf_{u^*\in\Usfer^*}\|\Lambda^*u^*\|=
\dist{\nullv^*}{\Lambda^*\Usfer^*}$
is known in variational analysis as dual Banach constant of $\Lambda$ and, together with
the primal Banach constant, i.e. the quantity $\Bconst{\Lambda}=\sup_{y\in\Usfer}\inf\{\|x\|:\ x\in
\Lambda^{-1}(y)\}=\sup_{y\in\Usfer}\dist{\nullv}{\Lambda^{-1}(y)}$, provides a
quantitative estimate for the property of $\Lambda\in\Lin(\X,\Y)$ to be
open at a linear rate, namely such that $\Lambda\Uball\supseteq\alpha\Uball$,
for some constant $\alpha>0$. Historically, a qualitative characterization of this property
was already established in the Banach-Schauder theorem, stating that $\Lambda$
is open at a linear rate iff it is an epimorphism. The modern development of
variational analysis complemented the statement of the above theorem adding that,
whenever this happens, then setting $\sur\Lambda=\sup\{\alpha>0:\
\Lambda\Uball\supseteq\alpha\Uball\}$, the following quantitative relations
are true
$$
  \Bconst{\Lambda}<+\infty,\qquad\dBconst{\Lambda}>0,\qquad
  \Bconst{\Lambda}\cdot\dBconst{\Lambda}=1,
$$
and
$$
  \sur\Lambda=\dBconst{\Lambda}={1\over \Bconst{\Lambda}}
$$
(see, for instance, \cite[Section 1.2.3]{Mord06}).
Thus, under the validity of the equality $(\ref{eq:minmax})$, a kind of
uniform openness at a linear rate for each fan $\mathcal{G}_x$ at points
$x\in [\charfun{F}{C}>0]$ implies $\flat_F>0$. This fact seems to reveal
a connection of the solvability and error bound theory for $(\IGE)$
problems with one of the possible manifestation of metric regularity,
a well-known property in variational analysis playing a key role
in the study of the solution stability of generalized equations of type
$(\GE)$ (see \cite{DonRoc14,Robi80}).
Connections of this type have started to be explored also in
\cite{Uder19}.
\end{remark}

An analogous scheme of analysis can be reproduced in the case $C$
is assumed to be a close, convex cone, leading to formulate a
condition for the positivity of $\derbom{F}$.

\vskip2cm


\end{document}